\documentclass[final,3p,times]{elsarticle}

\usepackage{float}
\usepackage{amssymb}
\usepackage{amsthm}
\usepackage{amscd}
\usepackage{bm}
\usepackage{amsmath}
\usepackage{amsfonts}
\usepackage{graphicx}
\newtheorem{theorem}{Theorem}[section]
\newtheorem{remark}[theorem]{Remark}
\newtheorem{lemma}[theorem]{Lemma}
\newtheorem{corollary}[theorem]{Corollary}
\newtheorem{proposition}[theorem]{Proposition}
\usepackage{mathrsfs}
\usepackage{titletoc}
\usepackage{color}

\usepackage{array,multirow,longtable,tabularx,booktabs}
\usepackage[flushleft]{threeparttable}
\usepackage{arydshln} 

\newcommand{\be}{\begin{equation}}

\newcommand{\ee}{\end{equation}}
\newcommand{\bea}{\begin{eqnarray}}
\newcommand{\eea}{\end{eqnarray}}
\newcommand{\bna}{\begin{eqnarray*}}
\newcommand{\ena}{\end{eqnarray*}}

\usepackage{geometry}
\journal{***}

\begin{document}

\begin{frontmatter}

\title{Nonlinear Diffusion Equations on Graphs: Global Well-Posedness, Blow-Up Analysis and Applications}

\author[qnu]{Mengqiu Shao}
\ead{mqshaomath@qfnu.edu.cn}
\address[qnu]{School of Mathematical Sciences, Qufu Normal University, Shandong, 273165, China}

\author[ruc]{Yunyan Yang}
\ead{yunyanyang@ruc.edu.cn}
\address[ruc]{School of Mathematics, Renmin University of China, Beijing, 100872, China}
\author[bnu]{Liang Zhao\corref{Zhao}}
\ead{liangzhao@bnu.edu.cn}
\address[bnu]{School of Mathematical Sciences, Key Laboratory of Mathematics and Complex Systems of MOE,\\
Beijing Normal University, Beijing, 100875, China}

\cortext[Zhao]{Corresponding author.}

\begin{abstract}
For a nonlinear diffusion equation on graphs whose nonlinearity violates the Lipschitz condition, we prove short-time solution existence and characterize global well-posedness by establishing sufficient criteria for blow-up phenomena and quantifying blow-up rates. These theoretical results are then applied to model complex dynamical networks, with supporting numerical experiments. This work mainly makes two contributions: (i) generalization of existing results for diffusion equations on graphs to cases with nontrivial potentials, producing richer analytical results; (ii) a new PDE approach to model complex dynamical networks, with preliminary numerical experiments confirming its validity.
\end{abstract}

\begin{keyword} Diffusion equation \sep complex network \sep blow-up \sep synchronization
 
 \vspace{4pt}

\end{keyword}

\end{frontmatter}

\section{Background}
\label{sec1}

Many natural and human-made systems—including ecosystems, electrical power grids, social networks, and biological as well as artificial neural networks—can be modeled as networks, where nodes represent the elements of the system and edges capture the pair-wise interactions between them. This framework has given rise to network science, an interdisciplinary research field that has attracted growing interest from researchers across diverse disciplines. A Network with such discrete structures is usually referred to as a graph in mathematics. Just as partial differential equations (PDEs) are used to model problems on continuous Euclidean spaces or manifolds, PDEs on discrete graphs should naturally become a powerful tool for studying network science and the phenomena associated with real-world networks. 

In recent years, PDEs on graphs have attracted significant attention and yielded a wealth of theoretical results. For example, Grigor'yan, Lin and Yang systematically raised and studied several nonlinear elliptic equations on graphs in \cite{grigor2016kazdan,grigor2016yamabe,grigor2017existence}.  Keller and Schwarz studied the Kazdan-Warner equation on canonically compactifiable graphs in \cite{keller2018kazdan}. Zhang and Zhao \cite{zhang2018convergence} studied the existence and convergence of solutions for some nonlinear Schr\"{o}dinger equations on a locally finite graph. Huang, Lin and Yau \cite{huang2020existence} established the existence of solutions to the mean field equations on finite graphs. Hua and Xu \cite{hua2023existence} investigated the existence of ground state solutions to some nonlinear Schr\"{o}dinger equations on lattice graphs. Li, Sun and Yang studied the Chern-Simons Higgs models on finite graphs by using the method of topological degree in \cite{li2024topological}. Hou and Kong \cite{hou2025existence} considered the existence and asymptotic behaviour of solutions to Chern-Simons systems and equations on finite graphs, etc.  

For modeling dynamic networks, approaches based on graph evolution equations are essential to encode both local interactions and global propagation. In this paper, we mainly focus on the following Cauchy problem
\begin{equation}\label{heatequation-1}
	\left\{\begin{array}{lll}
		\partial_t u=\Delta u-au+|u|^{p-1}u,&&x\in V,\,\, t>0,\\[1.2ex]
		u(x,0)=u_0(x),&&x\in V.
	\end{array}\right.
\end{equation}
The operator $\Delta$ is the $\mu$-Lapacian on $G$, where $G=(V,E,\mu,\omega)$ is a network that consists of $N$ nodes $V=\{x_1,x_2,\cdots x_N\}$. We use $e_{ij}\in E$ to denote an edge connecting nodes $x_i$ and $x_j$. The function $\mu: V\rightarrow \mathbb{R}^+$ is a positive measure on nodes and $\omega: E\rightarrow \mathbb{R}^+$ is a positive weight on edges. We always suppose that $G$ is connected and $\omega$ is symmetric, i.e., any two nodes in $V$ can be connected by a finite number of edges and $\omega_{ij}=\omega_{ji}$ for any $e_{ij}\in E$. In \eqref{heatequation-1}, we always assume that $p>1$, $a$ is a potential function on $V$ satisfying $a(x)\geq 0, \forall x\in V$, and $u_0(x)$ is the initial value. 

For parabolic equations of type \eqref{heatequation-1} on the Euclidean space $\mathbb{R}^n$, the study of solution properties, including well-posedness, asymptotic behaviour, and blow-up phenomena, constitutes a fundamental research program. When $a(x)\equiv 0$, the equation reduces to a Fujita-type equation, and results in \cite{fujita1966blowup} show that if $1<p<1+\frac{2}{n}$, the equation does not admit any nonnegative global solution. For a constant potential $a(x) \equiv C$, if $u_0$ decays fast enough \cite{lee1992global} or the initial energy is negative \cite{levine1974instability}, the solution will blow up in finite time. In addition, Zhang \cite{zhang2024blowup} established a blow-up criterion of $\mathbb{R}^{n}$ that depends on $p$ and $a(x)$.

The potential well method provides another framework for the above problem on $\mathbb{R}^n$, which was introduced by Payne and Sattinger in \cite{payne1975saddle}. In \cite{levine1974instability}, the non-existence of the global solution for
\begin{equation}\label{p1}
  \partial_tu-\Delta u=f(u),\ \ x\in \mathbb{R}^n
\end{equation}
is proved by this method. In \cite{ball1977remarks} a further result for \eqref{p1} was established, namely point-wise blow-up in finite time. Liu, Xu and Yu \cite{liu2008global} discussed the Cauchy problem for \eqref{p1} with $f(u)=|u|^{p-1}u-u$. By introducing a family of potential wells, they got the corresponding threshold results on the global existence, non-existence and asymptotic behaviour of solutions with initial energy $J(u_0)\leq d$, where $d$ is the potential well depth or mountain pass level. For more results, including one may refer to the monograph \cite{quittner2019superlinear}.

On graphs, most literature focuses on problem \eqref{heatequation-1} in the absence of potential functions (i.e., $a(x) \equiv0$). For example, Lin and Wu \cite{lin2017existence} established the existence and non-existence of global solutions for the following semilinear heat equation 
\begin{equation*}
  \left\{\begin{array}{lll}
		\partial_t u=\Delta u+u^{1+\alpha},&&x\in V,\,\, t>0,\\[1.2ex]
		u(x,0)=u_0(x),&&x\in V,
	\end{array}\right.
\end{equation*}
generalizing Fujita's results \cite{fujita1966blowup} to graphs via heat kernel estimates. Lin, Liu and Wu \cite{lin2025blowup} analyzed blow-up phenomenon for unbounded Laplacians. Meglioli \cite{meglioli2025uniqueness} proved uniqueness for heat equations with positive density on infinite graphs. Hua and Yang \cite{hua2025liouville}  derived Liouville theorems for ancient solutions with subexponential growth, extending Mosconi's manifold results \cite{mosconi2021liouville} to graphs.

Another noteworthy aspect of \eqref{heatequation-1} is its close connection to complex dynamical networks, which have attracted attention from many researchers in various fields. In particular, synchronization in complex networks has been a highly active area of research (see \cite{tang2014synchronization} for a comprehensive survey). Barahona and Pecora \cite{barahona2002synchronization} analyzed the stability of synchronization in small-world networks using the master stability function (MSF), which relates the stability to the spectral properties of the underlying network structure. Achieving pinning synchronization in complex networks has also been extensively studied (e.g., \cite{wang2002pinning,yu2009pinning,wang2014pinning,lu2014synchronization}). A rigorous connection between complex dynamical networks and the nonlinear diffusion equaiton (1) emerges when modeling a network of $N$ coupled nodes, where the state $\mathbf{u}_i$ of each node $x_i$ is  governed by an $n$-dimensional nonlinear system
 \begin{equation}\label{system1}
 	\dot{\mathbf{u}}_i =\mathbf{f}(\mathbf{u}_i)+\alpha\sum_{j=1}^{N}l_{ij}H\mathbf{u}_j+\mathbf{a}_i(\mathbf{u}_1, \mathbf{u}_2, \cdots, \mathbf{u}_N).
 \end{equation}
Here $i,j=1, 2, \cdots, N$, $\mathbf{u}_i=(u_{i1},u_{i2},\cdots, u_{in})\in \mathbb{R}^n$ is the state vector of node $x_i$, $\mathbf{f}(\cdot)$ describes the self-dynamics of each node, and the positive constant $\alpha$ denotes the coupling strength of the network. The $N\times N$ Laplacian matrix $L=(l_{ij})$ represents the coupling configuration of the network. For $i\neq j$, if there is an edge $e_{ij}\in E$, $l_{ij}=1$; otherwise, $l_{ij}=0$. The diagonal elements
 $l_ {ii}=-\sum\limits_{\substack{j=1\\ j\neq i}}^{N} l_{ij}=-k_i,$
 where $k_i$ is the degree of node $x_i$. The matrix $H=(h_{st})$ is an $n\times n$ state linking matrix, where $h_{st}\neq 0$ if two coupled nodes are linked through their $s$-th and $t$-th state variables, with $s,t=1,2,\cdots,n$. Finally, $\mathbf{a}_i$ is a controller applied at node $x_i$ and is to be designed.


The stabilization of system \eqref{system1} has been widely discussed \cite{tang2014synchronization}, which involves designing a proper controller $\mathbf{a}_i$ and corresponding conditions on the network topology such that the system can be controlled into an equilibrium state defined by
\begin{equation}
	\mathbf{u}_1=\mathbf{u}_2=\cdots=\mathbf{u}_N=\bar{\mathbf{u}},
\end{equation}
where $\mathbf{f}(\bar{\mathbf{u}})=\mathbf{0}$. A commonly used approach to deal with this problem is the Lyapunov method and the stability of the system is related to the eigenvalues of the Jacobian matrix and the master stability function of the system at an equilibrium state. This approach requires the nonlinear function $\mathbf{f}$ to be of Lipschitz type, i.e., there exist positive constants $C_1$ and $C_2$, such that
\begin{equation*}
	(\mathbf{y}-\mathbf{z})^T \left(\mathbf{f}(\mathbf{y})-\mathbf{f}(\mathbf{z})-C_1H\left(\mathbf{y}-\mathbf{z}\right)\right)\leq -C_2 \|\mathbf{y}-\mathbf{z}\|^2, \forall \mathbf{y}, \mathbf{z}\in\mathbb{R}^n.
\end{equation*}
However, for a general coupled network with its nonlinear function $\mathbf{f}$ not satisfying the Lipschitz condition, the evolution of system $\eqref{system1}$ depends on the initial condition of the system. The solution may blow up in finite time, and only when long time existence holds can the stability of the system be discussed. This situation has been less explored and represents a key focus of our study.

Inspired by these observations, this paper has three main objectives. (i) To establish short- and long-time existence of solutions to problem \eqref{heatequation-1} on graphs, focusing on the case of non-constant, non-vanishing potentials $a(x)$. (ii) To characterize asymptotic behaviour and blow-up criteria. (iii) Furthermore, we conduct several numerical experiments applying these theoretical results to complex dynamical networks and provide interpretations of the experimental results.

The rest of this paper is organized as follows. In Section 2, after introducing some useful definitions and notations, we present the main results of the paper. In Section 3, we prove both local and global existence of solutions. Furthermore, for solutions existing for all time, we establish their decay rates as $t\rightarrow +\infty$. In Section 4, using the potential well method, we study the existence of global solution and blow-up behaviour of solutions. In Section 5, we estimate the blow-up time of the solution of problem \eqref{heatequation-1}, followed by identifying several sufficient criteria that lead to solution blow-up and an analysis of the blow-up rate. In Section 6, we apply our results to complex dynamical networks.

\begin{remark}
	We note that Meng \cite{meng2025well} recently investigated an analogous reaction-diffusion system in her doctoral dissertation, deriving results on stability and asymptotic properties. However, our work addresses a more general class of systems with nontrivial potential functions, which yields finer descriptions of blow-up time and blow-up rates, and we also complete numerical experiments to characterize complex dynamical networks.
\end{remark}

\section{Main results}
\label{sec2}

To formulate the main results precisely, we begin by introducing necessary definitions and notation.
For any function $u:V\rightarrow \mathbb{R}$ and $x\in V$, the $\mu$-Laplacian  of $u$ at $x$ is defined by
\begin{equation*}
\Delta u(x):=\frac{1}{\mu(x)}\underset{y\sim x}\sum \omega_{xy}(u(y)-u(x)).
\end{equation*}
From now on, we always suppose that $p>1$, $a(x)\geq 0$ and $a(x)\not\equiv 0$ for all $x\in V$. Let $J:W^{1,2}(V)\rightarrow \mathbb{R}$ be an energy functional defined by
  \begin{equation}\label{en-funct}J(u)=\frac{1}{2}\int_V(|\nabla u|^2+au^2)d\mu-\frac{1}{p+1}\int_V|u|^{p+1}d\mu.
  \end{equation}
The Nehari functional $N:W^{1,2}(V)\rightarrow\mathbb{R}$ reads as
  \begin{equation}\label{N-def}
  N(u)=\int_V(|\nabla u|^2+au^2)d\mu-\int_V|u|^{p+1}d\mu.\end{equation}
The potential well associated with equation \eqref{heatequation-1}
  is defined by
  \begin{equation}\label{W-def}
  \mathscr{W}=\left\{u\in W^{1,2}(V): J(u)<{\textbf r},\,N(u)>0\right\}\cup\{0\},
  \end{equation}
  where $\textbf{r}$ is the depth of the potential well, given as
  \begin{equation}\label{r-def}{\textbf r}=\inf\left\{J(u):u\in W^{1,2}(V)\setminus\{0\},\,N(u)=0\right\}.\end{equation}
  The exterior of the potential well is the set
  $$\mathscr{S}=\{u\in W^{1,2}(V): J(u)<\textbf{r},\,N(u)<0\}.$$
We use $\lambda_a$ and $\varphi$ to denote the first eigenvalue and eigenfunction of the operator $-\Delta+a$, that is $-\Delta \varphi+a\varphi=\lambda_a\varphi$. 
 
 \vspace{4pt}
Our first result is as follows.

\begin{theorem}\label{short-time}
There exists some $T_0>0$ such that \eqref{heatequation-1}
has a unique solution $u(x,t)$ on $V\times [0,T_0]$. Moreover, if $u_0\geq 0$ on $V$, then $u(x,t)\geq 0$
for all $(x,t)\in V\times[0,T_0]$.
\end{theorem}

Based on the local existence in Theorem ~\ref{short-time}, we further investigate the global existence and asymptotic behaviour of solutions to \eqref{heatequation-1}.

\begin{theorem}\label{global}
(i) There
exist constants $\delta>0$, $\sigma>0$ and $C\geq 1$ such that
if $\|u_0\|_{L^\infty(V)}< \delta$, then \eqref{heatequation-1} has a global solution $u:V\times[0,+\infty)\rightarrow\mathbb{R}$
satisfying $$\|u(\cdot,t)\|_{L^\infty(V)}\leq C\|u_0\|_{L^\infty(V)} e^{-\sigma t}\ \ \hbox{for all}\ \ t>0.$$

(ii) If
$$\|u_0\|_{L^2(V)}<\epsilon_0:=\left(\frac{\lambda_a^2\mu_{\rm min}^{{p+1}}}{4|V|}\right)^{1/(p-1)},$$
then \eqref{heatequation-1}
has a unique solution $u:V\times [0,+\infty)\rightarrow \mathbb{R}$ satisfying
$$\|u(\cdot,t)\|_{L^\infty(V)}\leq \frac{1}{\sqrt{\mu_{\rm min}}}\|u_0\|_{L^2(V)}\,e^{-\frac{\lambda_a}{2} t}\ \ \hbox{for all}\ \ t\in[0,+\infty),$$
where $\mu_{\rm min}:=\min_{x\in V}\mu(x)$.
\end{theorem}

In addition, using the potential well method, we analyze the global existence and the blow-up phenomenon.

\begin{theorem}\label{Well-thm}
Suppose $T_{\rm max}(u_0)$ is the maximal time of existence for solutions to \eqref{heatequation-1}. 
\begin{enumerate}
\item [$(i)$] If $u_0\in\mathscr{W}$, then $T_{\rm max}(u_0)=+\infty$, $u(\cdot,t)\in\mathscr{W}$
  for all $t>0$, and $\|u(\cdot,t)\|_{L^\infty(V)}\rightarrow 0$ as $t\rightarrow +\infty$. 
\item [$(ii)$] If $u_0\in\mathscr{S}$, then
  $T_{\rm max}(u_0)<+\infty$. 
\end{enumerate}
\end{theorem}

Finally, we derive criteria for $u_0$ which guarantee blow-up of solutions in finite time. 

\begin{theorem}\label{finitetime}
Let $u:V\times [0,T_{\rm max}(u_0))\rightarrow \mathbb{R}$ be a solution of \eqref{heatequation-1} with $u_0\geq 0$, where  $[0,T_{\rm max}(u_0))$ is
the maximum time interval that the solution exists. Then the following assertions are true.\\
$(i)$ If
$$\int_Vu_0d\mu>c_1:=\left(\max_{x\in V}a(x)\right)^{{1}/{(p-1)}}|V|,$$
then
\begin{equation}\label{T-1}T_{\rm max}(u_0)\leq \frac{1}{(p-1)\max_{x\in V}a(x)}\log\frac{1}{1-|V|^{p-1}\max_{x\in V}a(x)
\left(\int_Vu_0d\mu\right)^{1-p}};\end{equation}
$(ii)$ If 
$$\int_V\varphi u_0d\mu>c_2:=\lambda_a^{\frac{1}{p-1}}\int_V\varphi d\mu,$$
then
\begin{equation}\label{T-2}T_{\rm max}(u_0)\leq \frac{1}{(p-1)\lambda_a}\log\frac{1}{1-\lambda_a
\left(\int_V\varphi d\mu\right)^{p-1}\left(\int_V\varphi u_0d\mu\right)^{1-p}},
\end{equation}
where $\lambda_a$ and $\varphi$ denote the first eigenvalue and eigenfunction of the operator $-\Delta+a$, respectively.
\end{theorem}

Under a further condition on the initial energy, we also investigate the upper bound estimates for the blow-up time.

\begin{theorem}\label{f-time1}
Let $u:V\times [0,T_{\rm max}(u_0))\rightarrow \mathbb{R}$ be a solution of \eqref{heatequation-1} with $u_0\geq 0$, where  $[0,T_{\rm max}(u_0))$ is
the maximum time interval that the solution exists. 
If $J(u_0)<0$, then
$$T_{\rm max}(u_0)\leq\frac{p+1}{(p-1)^2}|V|^{\frac{p-1}{2}}\|u_0\|^{1-p}_{L^{2}(V)};$$
While if
$$0\leq J(u_0)<c_3:=\frac{p-1}{2(p+1)}|V|^{\frac{1-p}{2}}\|u_0\|^{1+p}_{L^{2}(V)},$$
then
\begin{equation*}
T_{\max}(u_0)\leq\max\left\{\frac{(p+1)\left[\left(4(p+1)|V|^{\frac{p-1}{2}}J(u_0)\right)^{2/(p+1)}-(p-1)^{\frac{2}{p+1}}
\|u_0\|_{L^{2}(V)}^2\right]}
{(p-1)^{\frac{2}{p+1}}\left[2(p-1)|V|^{\frac{1-p}{2}}\|u_0\|^{1+p}_{L^{2}(V)}-4(p+1)J(u_0)\right]},
\frac{2(p+1)\|u_0\|^{1-p}_{L^{2}(V)}}{(p-1)^2|V|^{\frac{1-p}{2}}}
\right\}.
\end{equation*}
\end{theorem}

For the convenience of the readers, we summarize our main results in Table 1.

\begin{table}[H]
    \caption{Main results}
    \label{table:PPs}
    \vspace{3pt}  
    \small\centering
    \renewcommand\arraystretch{1.2} 
    \begin{tabular}{c l l c c l}  
        \toprule[1pt]
        Theorem  & Initial &  Existence & Asymptotic behaviour &Blow up  &  \\
        \midrule
        \ref{short-time}&   & Local existence &   &      &\\[3pt]
        \multirow{2}{*}{\ref{global}} &$||u_{0}||_{L^{\infty}(V)}<\delta$   &Global existence  & Exponential decay   &    &  \\[3pt]
        \multirow{2}{*}{} &$||u_{0}||_{L^{2}(V)}<\epsilon_{0}$   & Global existence  & Exponential decay      &   \\[3pt]
        \multirow{2}{*}{\ref{Well-thm}} &$u_{0}\in \mathscr{W}$   &Global existence  & Decay   &    &  \\[3pt]
        \multirow{2}{*}{} &$u_{0}\in \mathscr{S}$   &   &       &$\checkmark$ \\
        \multirow{2}{*}{\ref{finitetime}} &$u_{0}\geq 0$, $\int_{V}u_{0}d\mu>c_{1}$   &   &   &$\checkmark$    &  \\[3pt]
        \multirow{2}{*}{} &$u_{0}\geq 0$, $\int_{V}\varphi u_{0}d\mu>c_{2}$   &   &   &$\checkmark$    &   \\[5pt]
        \multirow{2}{*}{\ref{f-time1}} &$u_{0}\geq 0$, $J(u_{0})<0$   &   &   &$\checkmark$    &  \\[1pt]
        \multirow{2}{*}{} &$u_{0}\geq 0$, $0\leq J(u_{0})<c_{3}$   &   &   &$\checkmark$    &   \\[3pt]
        
       
        \bottomrule[1pt]
    \end{tabular}
\end{table}

Comparing with the existing literature, the main contributions of this paper are as follows.

(i) We generalize the classical results of nonlinear parabolic equations on the Euclidean space \cite{levine1974instability,fujita1966blowup,lee1992global,zhang2024blowup,quittner2019superlinear} to discrete graphs. Moreover, the equation \eqref{heatequation-1} with potential $a(x)\equiv 0$ considered in this work is more general and challenging than existing formulations on graphs. Specifically, if $a(x)\equiv 0$, problem (\ref{heatequation-1}) reduces to the case studied in references \cite{lin2017existence,lin2025blowup}.

(ii) We establish the connection between PDEs on graphs and complex dynamical networks. Based on this perspective, PDEs rather than ordinary differential systems can be employed to study complex dynamical networks. Since the nonlinear term $|u|^{p-1}u$ in \eqref{heatequation-1} does not satisfy the Lipschitz condition, the MSF approach used by Barahona and Pecora in \cite{barahona2002synchronization} is not applicable. We analyze the relationship between long-time existence, initial values, and network structure for \eqref{heatequation-1}, where the nonlinear term of the complex dynamical network does not satisfy the Lipschitz condition. Additionally, we explore the blow-up conditions of the network and provide estimates of the blow-up rate. As validation, we conduct corresponding numerical experiments, which verify the effectiveness of this approach.


\section{The existence of solutions }

In this section, we prove the existence and uniqueness of the short-time solutions to problem \eqref{heatequation-1} on a finite graph $G=(V,E)$. Additionally, under varying initial conditions, we establish the long-time existence of solutions and analyze their asymptotic behaviour. In preparation for the proof, we first establish the following comparison principle.

 \begin{lemma}\label{comparison-1}(Comparison principle)
If $u,v\in C^{1}(V\times [0,T])$ satisfy
\begin{equation*}
\left\{\begin{array}{lll}\partial_t u-\Delta u+au-|u|^{p-1}u\geq\partial_t v-\Delta v+av-|v|^{p-1}v,&&(x,t)\in V\times [0,T],\\[1.2ex]
u(x,0)\geq v(x,0),&&x\in V,
\end{array}\right.
\end{equation*}
there holds $u(x,t)\geq v(x,t)$ for all $(x,t)\in V\times[0,T]$, where $C^{1}(V\times [0,T])$ consists of all functions $u$ defined on $V\times [0,T]$ which satisfy $u(x,\cdot)\in C^{1}[0,T]$ for each $x\in V$.
\end{lemma}

\begin{proof} Similar to \cite{chung2011extinction}, we set $w(x,t)=v(x,t)-u(x,t)$, $w_+(x,t)=\max\{w(x,t),0\}$ and
 $w_-(x,t)=\min\{w(x,t),0\}$ for all $(x,t)\in V\times[0,T]$. Obviously,
 $w_+(\cdot,t)$ is a Lipschitz function for $t\in [0,T]$, as well as the function
 $$z(t)=\int_Vw_+^2(\cdot,t)d\mu.$$
 For all $x\in V$ and a.e. $t\in[0,T]$, we have
 $$\partial_tw_+^2(x,t)=2w_+(x,t)\partial_tw_+(x,t)=2w_+(x,t)\partial_tw(x,t).$$
Since
$$\partial_tw-\Delta w+aw-(|v|^{p-1}v-|u|^{p-1}u)\leq 0$$
and
$$\left||v|^{p-1}v-|u|^{p-1}u\right|\leq p\left(|v|^{p-1}+|u|^{p-1}\right)|v-u|,$$
 we obtain for a.e. $t\in[0,T]$,
 \begin{align*}
 z^\prime(t)&=2\int_Vw_+(\cdot,t)\partial_tw(\cdot,t)d\mu\\
 &\leq 2\int_Vw_+(\Delta w-aw)d\mu+2\int_Vw_+\left||v|^{p-1}v-|u|^{p-1}u\right|d\mu\\
 &\leq -2\int_V\nabla w_+\nabla wd\mu-2\int_Vaw_+^2d\mu+2p\int_Vw_+(|v|^{p-1}+|u|^{p-1})|w|d\mu.
 \end{align*}
 Noting that $-2\int_Vaw_+^2d\mu\leq 0$, $w_+|w|=w_+^2$ and
 \begin{align*}
 -2\int_V\nabla w_+\nabla w_-d\mu&=-2\sum_{x\in V}\mu(x)\sum_{y\sim x}\frac{\omega_{xy}}{2\mu(x)}(w_+(y,t)-w_+(x,t))(w_-(y,t)-w_-(x,t))\\
 &=\sum_{x\in V}\sum_{y\sim x}{\omega_{xy}}(w_+(y,t)w_-(x,t)+w_+(x,t)w_-(y,t))\\
 &\leq 0,
 \end{align*}
 we calculate for a.e. $t\in[0,T]$,
 $
 z^\prime(t)\leq Az(t),
 $
 and thus
 $$\left(e^{-At}z(t)\right)^\prime=e^{-At}(z^\prime(t)-Az(t))\leq 0,$$
  where $A=2p\max_{V\times[0,T]}(|u|^{p-1}+|v|^{p-1})$. Since the Newton-Leibnitz formula still holds
  for Lipschitz functions, it then follows that for all $t\in[0,T]$,
  $$e^{-At}z(t)-z(0)=\int_0^t\left(e^{-As}z(s)\right)^\prime ds\leq 0.$$
  This together with the facts $z(t)\geq 0$ and
  $z(0)=\int_Vw^2_+(\cdot,0)d\mu=0$
  leads to $z(t)\equiv 0$ for all $t\in[0,T]$. Therefore $u(x,t)\geq v(x,t)$ for all
  $(x,t)\in V\times[0,T]$. 
  \end{proof}

The following corollary follows immediately from the lemma above.

\begin{corollary}\label{cor1}
Suppose that $u\in C^{1}(V\times [0,T])$ satisfies
\begin{equation}\label{heatequation-1a}
\left\{\begin{array}{lll}\partial_t u=\Delta u-au+|u|^{p-1}u,&&(x,t)\in V\times [0,T]\\[1.2ex]
u(x,0)=u_0(x)\geq 0,&&x\in V.
\end{array}\right.
\end{equation}
Then $u\geq 0$ in $V\times [0,T]$.
\end{corollary}

At this stage, we are ready to prove Theorem \ref{short-time}.
\vspace{0.2cm}

{\emph{Proof of Theorem \ref{short-time}}.} Since $V=\{x_1,x_2,\cdots,x_N\}$, we can identify a function $u:V\times [0,+\infty)\rightarrow \mathbb{R}$
with a column vector $\bf{y}\in\mathbb{R}^N$, say ${\bf y}=(u(x_1,t),u(x_2,t),\cdots,u(x_N,t))^\top$. Define a map
$F:\mathbb{R}^N\rightarrow\mathbb{R}^N$ by
$$F({\bf y})=(F_1({\bf y}),F_2({\bf y}),\cdots,F_N({\bf y})),$$
where for each $j=1,2,\cdots,N$,
$$F_j({\bf y})=\Delta u(x_j,t)+|u(x_j,t)|^{p-1}u(x_j,t)-a(x_j)u(x_j,t).$$
As a consequence, the problem \eqref{heatequation-1} is thereby transformed into a system of ordinary differential equations
\begin{equation}\label{system}
\left\{\begin{array}{lll}{\bf y}^\prime(t)=F({\bf y}(t))\\[1.2ex]
{\bf y}(0)={\bf y}_0
\end{array}\right.\end{equation}
with the initial value is a vector ${\bf y}_0=(u_0(x_1),u_0(x_2),\cdots,u_0(x_N))^\top$. Since $p>1$, we have that
$F:\mathbb{R}^N\rightarrow\mathbb{R}^N$ is a locally Lipschitz map. According to the theory of ordinary differential equations (Section 6.1.1 in \cite{wang2006ordinary}), there exists 
$$T_0=\frac{1}{\max_{|{\bf y}-{\bf y}_0|\leq 1}|F({\bf y})|}>0$$
such that the system (\ref{system}) has a unique solution ${\bf y}(t)$ on the interval $[0,T_0]$. Here and in the sequel,
we denote the length of a vector in $\mathbb{R}^N$ by $|\cdot|$. It then follows that
equation (\ref{heatequation-1}) has a unique solution $u(x,t)$ on $V\times[0,T_0]$.
Moreover, if $u_0(x)\geq 0$ for all $x\in V$, then Lemma \ref{comparison-1}
implies that $u(x,t)\geq 0$ for all $(x,t)\in V\times[0,T_0]$.
$\hfill\Box$

\vspace{0.2cm}
Next, we introduce a key lemma that will be used in the proof of Theorem~\ref{global}.

\begin{lemma}\label{lambda}
The first eigenvalue $\lambda_a$ of the operator $-\Delta+a$ is positive and there exists an eigenfunction
$\varphi$ such that
$$\left\{\begin{array}{lll}-\Delta \varphi+a\varphi=\lambda_a\varphi\\[1.2ex]
\varphi(x)>0,\quad\,\,\forall x\in V.\end{array}\right.$$
\end{lemma}

\begin{proof}
  
Since
$$\lambda_a=\inf_{u\geq 0,\,u\not\equiv 0}\frac{\int_V(|\nabla u|^2+au^2)d\mu}{\int_Vu^2d\mu}=
\inf_{\int_Vu^2d\mu=1,\,u\geq 0}\int_V(|\nabla u|^2+au^2)d\mu,$$
we can take a minimizing sequence $\{u_k\}$ such that $\int_Vu_k^2=1$, $u_k\geq 0$ and
$$\int_V(|\nabla u_k|^2+au_k^2)d\mu\rightarrow \lambda_a$$
as $k\rightarrow +\infty$. Clearly, $\{u_k\}$ is uniformly bounded in $V$. Up to a subsequence (still denote by $\{u_k\}$), there exists a function
$\varphi : V\rightarrow \mathbb{R}$ such that $u_k$ uniformly converges to $\varphi$. This leads to $\int_V \varphi^2d\mu=1$, $\varphi\geq 0$ and
$$\lambda_a=\int_V(|\nabla \varphi|^2+a\varphi^2)d\mu>0.$$
As a consequence, $\varphi$ satisfies the Euler-Lagrange equation
$$-\Delta \varphi+a\varphi=\lambda_a\varphi,\quad\forall x\in V.$$
It remains to prove that $\varphi>0$ on $V$. Suppose not, since $\varphi$ is nonnegative, we may assume $\varphi(x_0)=0=\min_V\varphi$ for some $x_0\in V$. Thus
 \begin{eqnarray*}
 0=\lambda_a\varphi(x_0)=-\Delta \varphi(x_0)+a(x_0)\varphi(x_0)=-\Delta \varphi(x_0)\leq 0,
 \end{eqnarray*}
 which implies $\varphi(y)=0$ for all $y\sim x_0$. Since $G=(V,E)$ is finite and connected, repeating this process for finite times, we conclude $\varphi(x)=0$ for all $x\in V$, which contradicts $\int_Vu^2d\mu=1$. Therefore
 $\varphi(x)>0$ for all $x\in V$.
\end{proof}

At the end of this section, we prove Theorem~\ref{global}.
\vspace{0.2cm}

{\emph{Proof of Theorem \ref{global}}.} 
 $(i)$ By Lemma~\ref{lambda}, $\lambda_a>0$ and there exists an eigenfunction
$\varphi$ with $\varphi(x)>0$ for all $x\in V$.
Given any initial data $u_0$ with $\|u_0\|_{L^\infty(V)}<\delta$, where $\delta>0$ is a small number to be determined later. If $u_0\equiv 0$ on $V$, then (\ref{heatequation-1}) has a unique solution
$u(x,t)\equiv 0$ on $V\times[0,+\infty)$, and the assertion $(i)$ already holds. We next consider the case $u_0\not\equiv 0$. To proceed, we set for some positive number
$\sigma\in(0,\lambda_a)$,
$$v(x,t)=\frac{2\|u_0\|_{L^\infty(V)}}{\min_{x\in V}\varphi}e^{-\sigma t}\varphi(x).$$
Taking $\delta$ satisfying
$$\left(2\delta\frac{\max_{x\in V}\varphi}{\min_{x\in V}\varphi}\right)^{p-1}\leq \lambda_a-\sigma,$$
we have for all $(x,t)\in V\times[0,+\infty)$,
\begin{equation}\label{vxt}0\leq v(x,t)\leq (\lambda_a-\sigma)^{\frac{1}{p-1}}.\end{equation}
Clearly $v$ satisfies the equation
$$\left\{\begin{array}{lll}
\partial_tv-\Delta v+av=(\lambda_a-\sigma)v\\[1.2ex]
v(x,0)=\frac{2\|u_0\|_{L^\infty(V)}}{\min_V\varphi}\varphi(x).
\end{array}\right.$$
In view of (\ref{vxt}), we have $(\lambda_a-\sigma)v\geq v^p$. Thus on $V\times[0,+\infty)$,
$$\partial_tv-\Delta v+av\geq v^p$$
and on $V\times[0,T_{\rm max}(u_0))$,
\begin{equation}\label{com-2}\partial_tv-\Delta v+av-v^p\geq \partial_tu-\Delta u+au-|u|^{p-1}u.\end{equation}
Note that for all $x\in V$,
\begin{equation}\label{initial}v(x,0)>u_0(x)\end{equation}
and
\begin{equation}\label{initial-2}v(x,0)>-u_0(x). \end{equation}
Note also that $-u$ is also a solution of (\ref{heatequation-1}) with the initial data $-u_0$.
We now {\it claim} that $T_{\rm max}(u_0)=+\infty$. If not, we have $T_{\rm max}(u_0)<+\infty$. In view of (\ref{com-2}) and
(\ref{initial}),
we have by the comparison principle (Lemma \ref{comparison-1}), $v(x,t)\geq u(x,t)$ for all $(x,t)\in V\times[0,T_{\rm max}(u_0))$. While (\ref{com-2}),
(\ref{initial-2}) and the comparison principle imply $v(x,t)\geq -u(x,t)$ for all $(x,t)\in V\times[0,T_{\rm max}(u_0))$.
It then follows that for all $(x,t)\in V\times[0,T_{\rm max}(u_0))$,
$$|u(x,t)|\leq v(x,t)=\frac{2\|u_0\|_{L^\infty(V)}}{\min_V\varphi}e^{-\sigma t}\varphi(x)\leq C\|u_0\|_{L^\infty(V)} e^{-\sigma t}.$$
By the ODE theory (\cite{wang2006ordinary}, Chapter 6), the solution $u$ can be extended to $V\times [0,T_1)$ for some
$T_1>T_{\rm max}(u_0)$, contradicting the definition of $T_{\rm max}(u_0)$. This confirms our claim $T_{\rm max}(u_0)=+\infty$ and
$$\|u(\cdot,t)\|_{L^\infty(V)}\leq C\|u_0\|_{L^\infty(V)} e^{-\sigma t}$$
for all $t\in[0,+\infty)$. This completes the proof of the first assertion.

\vspace{0.2cm}
$(ii)$ We first prove that there exists two positive numbers $\epsilon$ and $\theta$, depending only on the graph $G=(V,E)$ and $p$, such that if $\|u_0\|_{L^2(V)}<\epsilon$, we
have \begin{equation}\label{eq-4}\|u(\cdot,t)\|_{L^2(V)}\leq \|u_0\|_{L^2(V)}e^{-\theta t},\quad\forall t\in[0,T_{\rm max}(u_0)).\end{equation}

To see this, multiplying both sides of (\ref{heatequation-1}) by $u$, we have by integration by parts
\begin{align*}
\frac{d}{dt}\int_Vu^2(\cdot,t)d\mu&=-2\int_V(|\nabla u|^2+au^2)d\mu+2\int_V|u|^{p+1}d\mu\\
&\leq-2\lambda_a\int_Vu^2(\cdot,t)d\mu+\frac{2|V|}{\mu_{\rm min}^{\frac{p+1}{2}}}\left(\int_Vu^2(\cdot,t)d\mu\right)
^{\frac{p+1}{2}},
\end{align*}
where $\mu_{\rm min}=\min_{x\in V}\mu(x)$ and we have use the fact that
$$\|u(\cdot,t)\|_{L^\infty(V)}\leq \frac{1}{\sqrt{\mu_{\rm min}}}\left(\int_Vu^2(\cdot,t)d\mu\right)^{1/2}.$$
This leads to
$$\left\{
\begin{array}{lll}
z^\prime(t)\leq-2\lambda_az(t)+\frac{2|V|}{\mu_{\rm min}^{\frac{p+1}{2}}}z^{\frac{p+1}{2}}(t),\\[1.2ex]
z(t)=\int_Vu^2(\cdot,t)d\mu,\quad t\in[0,T_{\rm max}(u_0)),\\[1.2ex]
z(0)=\int_Vu_0^2d\mu.
\end{array}
\right.$$
  Suppose that $\|u_0\|_{L^2(V)}<\epsilon$, where $\epsilon$ is a small positive number to be determined later. Set
  $$t_\epsilon=\sup\left\{t: \int_Vu^2(\cdot,\tau)d\mu<\epsilon^2,\,\forall \tau\in[0,t]\subset[0,T)\right\}.$$
  Obviously $t_\epsilon>0$. If we choose $\epsilon$ such that
  $$0<\epsilon\leq \epsilon^{\frac{1}{2}}_{0}=\left(\frac{\lambda_a\mu_{\rm min}^{\frac{p+1}{2}}}{2|V|}\right)^{1/(p-1)},$$
here $\epsilon_{0}$ is given in Theorem~\ref{global} (ii),  then it follows that
  \begin{align*}
  z^\prime(t)&\leq -2\left(\lambda_a-\frac{|V|}{\mu_{\rm min}^{\frac{p+1}{2}}}\epsilon^{{p-1}}\right)z(t)\\
  &\leq-\lambda_a z(t)
  \end{align*}
  and
  \begin{equation}\label{eq-3}z(t)\leq z(0)e^{-\lambda_a t},\quad\forall t\in [0,t_\epsilon).\end{equation}
  Now we {\it claim} that $t_\epsilon=T_{\rm max}(u_0)$.
  For otherwise, in view of the definition of $t_\epsilon$, there must hold $0<t_\epsilon<T_{\rm max}(u_0)$ and $z(t_\epsilon)=\epsilon^2$. Noticing
  (\ref{eq-3}), one has
  $$\epsilon^2=z(t_\epsilon)\leq z(0)e^{-\lambda_at_\epsilon}\leq z(0)<\epsilon^2,$$
  which is impossible. Hence our claim follows. Furthermore, the fact $t_\epsilon=T_{\rm max}(u_0)$ together with (\ref{eq-3}) implies (\ref{eq-4})
  with $\theta=\lambda_a/2$, as we desired.

  To sum up, we come to the conclusion
  $$\|u(\cdot,t)\|_{L^2(V)}\leq \|u_0\|_{L^2(V)}e^{-\frac{\lambda_a}{2} t},\quad\forall t\in[0,T_{\rm max}(u_0)).$$
  By the ODE theory (\cite{wang2006ordinary}, Chapter 6), $u(x,t)$ can be extended to $V\times[0,+\infty)$. Moreover,
  $$\|u(\cdot,t)\|_{L^\infty(V)}\leq \frac{1}{\sqrt{\mu_{\rm min}}}\|u(\cdot,t)\|_{L^2(V)}\leq \frac{1}{\sqrt{\mu_{\rm min}}}\|u_0\|_{L^2(V)}e^{-\frac{\lambda_a}{2}t}$$
  for all $t\in[0,+\infty)$.
  This ends the proof of $(ii)$, and thus completes the proof of the theorem.
$\hfill\Box$

 \section{The potential well theory}

In this section, we investigate the long-time existence and the blow-up behaviour for problem \eqref{heatequation-1} using the potential well method. First, we present some properties of the potential well.

\begin{lemma}\label{depth}
  Let 
  \begin{equation}\label{LAM}\Lambda=\inf\left\{\frac{\int_V(|\nabla u|^2+au^2)d\mu}{\left(\int_V|u|^{p+1}d\mu\right)^{2/(p+1)}}:
  u\in W^{1,2}(V)\setminus\{0\}\right\}.\end{equation}
  Then there hold the following three assertions:\\
  $(i)$ The depth $\emph{{\textbf{r}}}$ of the potential well $\mathscr{W}$ is attained by some function $u\in W^{1,2}(V)\setminus\{0\}$
  satisfying $N(u)=0$, moreover it is uniquely determined by $p$ and $\Lambda$, namely
  \begin{equation}\label{r}{\emph{\textbf{r}}}=\frac{p-1}{2(p+1)}\Lambda^{\frac{p+1}{p-1}},
  \end{equation}
where  $\mathscr{W}$ and $\emph{\textbf{r}}$ are defined in \eqref{W-def} and \eqref{r-def}.
  
  $(ii)$ For any $\epsilon>0$, we have
  \begin{equation}\label{d-e}
  \emph{\textbf{r}}_\epsilon=\inf\{J(u): u\in W^{1,2}(V),\,N(u)=-\epsilon\}\geq {\textbf{r}}-\frac{\epsilon}{p+1}.
  \end{equation}
  (iii) For any $u\in W^{1,2}(V)$, if $\|u\|_{1,a}<\sqrt{2\emph{\textbf{r}}}$, then $u\in\mathscr{W}$;
  while if $u\in\mathscr{W}$, then
  $$\|u\|_{1,a}<\sqrt{\frac{2(p+1)\emph{\textbf{r}}}{p-1}},$$
  where $\|u\|_{1,a}=(\int_V(|\nabla u|^2+au^2)d\mu)^{1/2}$.
  \end{lemma}
  \proof Firstly we prove $(i)$ and $(ii)$. By a direct method of variation, one can easily see that $\Lambda$ is attained by some function $u_\Lambda\not\equiv 0$,
  and thus $\Lambda>0$. Let $\epsilon\geq 0$ be fixed. We first {\it claim} that
  \begin{equation}\label{set}\left\{v\in W^{1,2}(V)\setminus\{0\}: N(v)=-\epsilon\right\}\not=\varnothing.\end{equation}
  To see this, for any $v\not\equiv 0$, we set a function $\phi(s)=N(sv)$. Since $\phi(0)=0$ and
  $$\phi(s)=s^2\int_V(|\nabla v|^2+av^2)d\mu-s^{p+1}\int_V|v|^{p+1}d\mu\rightarrow-\infty$$
  as $s\rightarrow+\infty$, there should exist some $s\in(0,+\infty)$ such that $\phi(s)=-\epsilon$ if $\epsilon>0$. In the case
  $\epsilon=0$, since $\phi(s_0)>0$ for some $0<s_0<1$, we can take some $s_1>t_0$ such that $\phi(s_1)=0$. This concludes our claim
  (\ref{set}). Now taking any $u$ in the set (\ref{set}), we have
  $$J(u)=\left(\frac{1}{2}-\frac{1}{p+1}\right)\int_V|u|^{p+1}d\mu-\frac{\epsilon}{2}.$$
  From (\ref{LAM}) and $N(u)\leq 0$, it follows that
  $$\int_V(|\nabla u|^2+au^2)d\mu\leq \int_V|u|^{p+1}d\mu\leq \Lambda^{-\frac{p+1}{2}}\left(\int_V(|\nabla u|^2+au^2)d\mu\right)^{\frac{p+1}{2}}$$
  and that
  $$\int_V(|\nabla u|^2+au^2)d\mu\geq \Lambda^{\frac{p+1}{p-1}}.$$
  Hence
  \begin{align}\nonumber
  J(u)&=\frac{1}{2}\|u\|_{1,a}^2-\frac{1}{p+1}\int_V|u|^{p+1}d\mu\\\nonumber
  &=\left(\frac{1}{2}-\frac{1}{p+1}\right)\|u\|_{1,a}^2-\frac{\epsilon}{p+1}\\
  &\geq\frac{p-1}{2(p+1)}\Lambda^{\frac{p+1}{p-1}}-\frac{\epsilon}{p+1}.\label{r-c}
  \end{align}
  Recall the definition of $\textbf{r}$, namely (\ref{r-def}). On one hand, in the case $\epsilon=0$, it follows from (\ref{r-c}) that
  \begin{equation}\label{r-g}\textbf{r}\geq \frac{p-1}{2(p+1)}\Lambda^{\frac{p+1}{p-1}};\end{equation}
  while in the case $\epsilon>0$, one has
  $$\textbf{r}_\epsilon\geq \frac{p-1}{2(p+1)}\Lambda^{\frac{p+1}{p-1}}-\frac{\epsilon}{p+1},$$
  which is exactly (\ref{d-e}). On the other hand, we can take a minimizing sequence $\{u_k\}$ for \eqref{LAM}, namely $u_k\not\equiv 0$ and
  \begin{equation}\label{Sobolev}
  	\frac{\|u_k\|_{1,a}^2}{\|u_k\|_{L^{p+1}(V)}^2}=\Lambda+o_k(1),
  \end{equation}
  where $o_k(1)\rightarrow 0$ as $k\rightarrow+\infty$. Using the same argument of proving (\ref{set}), one may find some $s_k>0$ such that
  $N(s_k u_k)=0$ for each $k$. Since $s_k u_k$ also satisfies (\ref{Sobolev}), one may assume without loss of generality that $N(u_k)=0$. As a consequence,
  $$\|u_k\|_{1,a}^2=\|u_k\|_{L^{p+1}(V)}^{p+1}=(\Lambda+o_k(1))^{-\frac{p+1}{2}}\|u_k\|_{1,a}^{p+1}$$
  and
  $$
  J(u_k)=\left(\frac{1}{2}-\frac{1}{p+1}\right)\|u_k\|_{1,a}^2=\frac{p-1}{2(p+1)}\Lambda^{\frac{p+1}{p-1}}+o_k(1).
  $$
  This implies $\textbf{r}\leq \frac{p-1}{2(p+1)}\Lambda^{\frac{p+1}{p-1}}$, which together with (\ref{r-g}) leads to (\ref{r}).

\vspace{4pt}
  Now we prove that as an infimum in (\ref{r-def}), $\textbf{r}$ can be attained. Similarly as above, we take a function sequence $\{v_k\}\subset W^{1,2}(V)\setminus\{0\}$ satisfying
  $N(v_k)=0$ and $J(v_k)=\textbf{r}+o_k(1)$, or equivalently
  $$
  \textbf{r}+o_k(1)=\frac{1}{2}\|v_k\|_{1,a}^2-\frac{1}{p+1}\|v_k\|_{L^{p+1}(V)}^{p+1}=\frac{p-1}{2(p+1)}\|v_k\|_{1,a}^2.
  $$
  Thus $\{v_k\}$ is uniformly bounded in $V$. Up to a subsequence, $\{v_k\}$ converges to some $v$ uniformly in $V$. Therefore $N(v)=0$ and $J(v)=\textbf{r}$. Since $\textbf{r}>0$, there holds $v\not\equiv 0$.
  Thus $\textbf{r}$ is attained.\\

  Now we prove $(iii)$. If $\|u\|_{1,a}<\sqrt{2\textbf{r}}$, then either $u\equiv 0$, or $u\not\equiv 0$ and
  $$J(u)=\frac{1}{2}\|u\|_{1,a}^2-\frac{1}{p+1}\int_V|u|^{p+1}d\mu\leq \frac{1}{2}\|u\|_{1,a}^2<\textbf{r}.$$
  In the case $u\not\equiv 0$, we have by combining (\ref{LAM}) and (\ref{r}),
  $$\int_V|u|^{p+1}d\mu\leq \Lambda^{-\frac{p+1}{2}}\|u\|_{1,a}^{p+1}\leq \Lambda^{-\frac{p+1}{2}}(2\textbf{r})^{\frac{p-1}{2}}
  \|u\|_{1,a}^2<\|u\|_{1,a}^2,$$
  which implies $u\in\mathscr{W}$.

  On the other hand, if $u\in\mathscr{W}$, then
  \begin{eqnarray*}
  \left(\frac{1}{2}-\frac{1}{p+1}\right)\|u\|_{1,a}^2<J(u)<\textbf{r}.
  \end{eqnarray*}
  This leads to $\|u\|_{1,a}<\sqrt{2(p+1)\textbf{r}/(p-1)}$, as we desired. $\hfill\Box$\\


Next, we give the proof of Theorem~\ref{Well-thm}.
\vspace{0.2cm}

 {\emph{Proof of Theorem \ref{Well-thm}}.} $(i)$ Suppose $u_0\in\mathscr{W}$. We have either $u_0\equiv 0$, or $J(u_0)<\textbf{r}$ and $N(u_0)>0$. We employ a case analysis with three possibilities
  
\vspace{0.2cm}
  {\it Case $1$.} $u_0\equiv 0$. In this case, the uniqueness theorem for system implies $u(\cdot,t)\equiv 0$ for all $t\in [0,+\infty)$.
  Thus the conclusion is true.
  
\vspace{0.2cm}
  {\it Case $2$.} There exists some $t_0>0$ such that $u(\cdot,t_0)\equiv 0$ and $u(\cdot,t)\not\equiv 0$ for all $t\in[0,t_0)$.
\vspace{0.2cm}

   The same reason as in Case 1 leads to $u(\cdot,t)\equiv 0$  for all $t\geq t_0$. In particular $T_{\rm max}(u_0)=+\infty$,
   $u(\cdot,t)\in\mathscr{W}$ for all $t\in[t_0,+\infty)$ and
   $\|u(\cdot,t)\|_{L^\infty(V)}\rightarrow 0$ as $t\rightarrow +\infty$. We also need to consider the situation on the interval $[0,t_0)$.
    In view of (\ref{heatequation-1}) and (\ref{en-funct}), we have
  \begin{align}\nonumber
  \frac{d}{dt}J(u(\cdot,t))&=\int_V(\nabla u\nabla u_t+auu_t-|u|^{p-1}uu_t)d\mu\\\nonumber
  &=\int_V(-\Delta u+au-|u|^{p-1}u)u_td\mu\\\label{derivative}
  &=-\int_Vu_t^2d\mu.
  \end{align}
  It follows that $J(u(\cdot,t))$ is decreasing in $t\in[0,t_0)$, in particular $J(u(\cdot,t))\leq J(u_0)<\textbf{r}$
  for all $t\in[0,t_0)$. Since $N(u_0)>0$ and $N(u(\cdot,t))$ is continuous in $t$, there holds $N(u(\cdot,t))>0$ if
  $t$ is sufficiently close to $0$. Now we {\it claim} that $N(u(\cdot,t))>0$ for all $t\in[0,t_0)$. For otherwise, there would
  be a $t_1\in(0,t_0)$ satisfying $N(u(\cdot,t_1))=0$, which together with (\ref{r-def}) gives $J(u(\cdot,t_1))\geq \textbf{r}$,
  which is a contradiction. Hence our claim follows, and thus $u(\cdot,t)\in\mathscr{W}$ for all $t\in[0,t_0)$, as we desired.

\vspace{0.2cm}
  {\it Case $3$.} $u(\cdot,t)\not\equiv 0$ for all $t\in[0,T_{\rm max}(u_0))$.

  Reasoning as in Case 2, we obtain $J(u(\cdot,t))\leq J(u_0)<\textbf{r}$ and $N(u(\cdot,t))>0$ for all
  $t\in[0,T_{\rm max}(u_0))$. Coming back to (\ref{W-def}), we conclude $u(\cdot,t)\in\mathscr{W}$ for all
  $t\in[0,T_{\rm max}(u_0))$. From Lemma \ref{depth} $(iii)$, it follows that
  $$\|u(\cdot,t)\|_{1,a}<\sqrt{\frac{2(p+1)\textbf{r}}{p-1}},\quad\forall t\in[0,T_{\rm max}(u_0)).$$
  Hence $u(\cdot,t)$ is uniformly bounded in $W^{1,2}(V)$, and also uniformly bounded in $L^\infty(V)$. Then
  the extension theorem implies $T_{\rm max}(u_0)=+\infty$. There is
  \begin{equation}\label{sta-1}\|u(\cdot,t)\|_{L^\infty(V)}\rightarrow 0\quad {\rm as}\quad t\rightarrow+\infty\end{equation}
  left to be proved.
  
  Since $u(\cdot,t)$ is uniformly bounded, $J(u(\cdot,t))$ is a bounded univariate function, say there exists
  a positive constant $C$ such that
  $$J(u(\cdot,t))=\frac{1}{2}\|u(\cdot,t)\|_{1,a}^2-\frac{1}{p+1}\int_V|u|^2d\mu\geq -C,\quad\forall t\in[0,+\infty).$$
  Integrating (\ref{derivative}) with respect to $t$ from $0$ to $+\infty$, we have
  $$\int_0^{+\infty}\int_Vu_t^2d\mu dt\leq J(u_0)+C.$$
  Hence there exists $t_k\rightarrow+\infty$ verifying
  \begin{equation}\label{tend-0}u_t(x,t_k)\rightarrow 0
  \,\,{\rm as}\,\,k\rightarrow+\infty\,\,{\rm uniformly\,\,in}\,\,x\in V.\end{equation}
  Again, the uniform boundedness of $u(x,t_k)$ implies that up to a subsequence, there exists some function $v$ on $V$ such that
  \begin{equation}\label{tend-01}u(x,t_k)\rightarrow v(x)
  \,\,{\rm as}\,\,k\rightarrow+\infty\,\,{\rm uniformly\,\,in}\,\,x\in V\end{equation}
  and that
  \begin{equation}\label{Ju}J(v)=\lim_{k\rightarrow+\infty}J(u(\cdot,t_k))\leq J(u_0)<\textbf{r}.\end{equation}
  Observing that the equation (\ref{heatequation-1}) at $(x,t_k)$ reads as
  $$u_t(x,t_k)=\Delta u(x,t_k)-a(x)u(x,t_k)+|u(x,t_k)|^{p-1}u(x,t_k)$$
  and passing to the limit $k\rightarrow+\infty$, we obtain by (\ref{tend-0}) and (\ref{tend-01}),
  $$\Delta v(x)-a(x)v(x)+|v(x)|^{p-1}v(x)=0,\quad x\in V.$$
  Multiplying the above equation by $v$, we have by integration by parts
  \begin{equation}\label{Nv}N(v)=\|v\|_{1,a}^2-\int_V|v|^{p+1}d\mu=0.\end{equation}
  Combining (\ref{r-def}), (\ref{Ju}) and (\ref{Nv}), we conclude
  \begin{equation}\label{v0}v(x)\equiv 0,\quad x\in V.\end{equation}
  Let us come back to (\ref{sta-1}). Suppose it does not hold. Then there exist a positive constant $\epsilon_0$ and an increasing
  number sequence $s_k\rightarrow+\infty$ such that
  $\lim_{k\rightarrow +\infty}\|u(\cdot,s_k)\|_{L^\infty(V)}=\zeta_0$.
  Up to a subsequence, we may assume $u(\cdot,s_k)\rightarrow w$ uniformly in $V$ as $k\rightarrow+\infty$.
  Obviously we have
  \begin{equation}\label{contr}\|w\|_{L^\infty(V)}=\zeta_0>0.\end{equation}
  The monotonicity and
  boundedness of $J(u(\cdot,t))$ together with (\ref{v0}) lead to
  $$J(w)=\lim_{k\rightarrow+\infty}J(u(\cdot,s_k))=\lim_{k\rightarrow+\infty}J(u(\cdot,t_k))=J(v)=0.$$
  This together with (\ref{contr}) (in particular $w\not\equiv 0$) gives
  $$\|w\|_{1,a}^2=\frac{2}{p+1}\int_V|w|^{p+1}d\mu<\int_V|w|^{p+1}d\mu,$$
  which contradicts
  $$N(w)=\lim_{k\rightarrow+\infty}N(u(\cdot,s_k))\geq 0,$$
  since $u(\cdot,t)\in\mathscr{W}$ for all $t\in[0,+\infty)$.  Therefore $(\ref{sta-1})$ holds and the proof of $(i)$
  is completed.\\

  $(ii)$ Let $u_0\in\mathscr{S}$, i.e. $J(u_0)<\textbf{r}$ and $N(u_0)<0$. We shall prove $T_{\rm max}(u_0)<+\infty$.
  For otherwise, $T_{\rm max}(u_0)=+\infty$. Take a number $\epsilon$ satisfying
  $$0<\epsilon<\min\left\{-N(u_0),\textbf{r}-J(u_0)\right\}.$$
  By (\ref{d-e}) and the monotonicity of $J(u(\cdot,t))$, we have for all $t\in[0,+\infty)$,
  $$J(u(\cdot,t))\leq J(u_0)<\textbf{r}-\epsilon<\textbf{r}-\frac{\epsilon}{p+1}\leq\textbf{r}_\epsilon.$$
  Since $N(u_0)<-\epsilon$, it follows from the definition of $\textbf{r}_\epsilon$ in (\ref{d-e}) and the continuity
  of $N(u(\cdot,t))$ that $N(u(\cdot,t))<-\epsilon$ for all $t\in[0,+\infty)$. Hence by (\ref{N-def}) and (\ref{heatequation-1}),
  one calculates
  $$\frac{d}{dt}\int_Vu^2(\cdot,t)d\mu=2\int_Vuu_td\mu=-N(u(\cdot,t))>\epsilon.$$
  Set $y(t)=\int_Vu^2(\cdot,t)d\mu$. It then follows that
  \begin{equation}\label{infty}
  y(t)\rightarrow+\infty\quad{\rm as}\quad t\rightarrow+\infty.
  \end{equation}
  On the other hand, the H\"older's inequality implies
  $$y(t)\leq \left(\int_V|u(\cdot,t)|^{p+1}d\mu\right)^{\frac{2}{p+1}}|V|^{\frac{p-1}{p+1}},$$
  and thus
  \begin{equation}\label{Hold}
  \int_V|u(\cdot,t)|^{p+1}d\mu\geq |V|^{-\frac{p-1}{2}}y^{\frac{p+1}{2}}(t),\quad\forall t\geq 0.
  \end{equation}
  In view of (\ref{infty}), there exists a sufficiently large $t_0>0$ such that
  \begin{equation}\label{large-t0}
  \frac{p}{p+1}|V|^{-\frac{p-1}{2}}y^{\frac{p+1}{2}}(t)>2J(u_0),\quad\forall t\geq t_0.
  \end{equation}
  Combining (\ref{heatequation-1}), (\ref{en-funct}), (\ref{Hold}) and (\ref{large-t0}), we obtain for
  $t\in[t_0,+\infty)$,
  \begin{align*}
  \frac{d}{dt}y(t)&=-2\|u(\cdot,t)\|_{1,a}^2+2\int_V|u|^{p+1}d\mu\\
  &=-2J(u(\cdot,t))+\frac{2p}{p+1}\int_V|u|^{p+1}d\mu\\
  &\geq-2J(u_0)+\frac{2p}{p+1}|V|^{-\frac{p-1}{2}}y^{\frac{p+1}{2}}(t)\\
  &\geq\frac{p}{p+1}|V|^{-\frac{p-1}{2}}y^{\frac{p+1}{2}}(t),
  \end{align*}
  and whence
  $$
  \frac{2}{1-p}y^{\frac{1-p}{2}}(t)-\frac{2}{1-p}y^{\frac{1-p}{2}}(t_0)\geq \frac{p}{p+1}|V|^{-\frac{p-1}{2}}(t-t_0).
  $$
  This together with (\ref{infty}) gives a contradiction. Therefore $T_{\rm max}(u_0)<+\infty$.
  $\hfill\Box$

\section{Estimation of blow-up time}

In this section, we focus on estimating the blow-up time and blow-up rate of solutions to \eqref{heatequation-1} with $u_0\geq 0$. First, we give the proof of Theorem~\ref{finitetime}.
 
 \vspace{0.2cm}

{\emph{Proof of Theorem \ref{finitetime}}.} 
 $(i)$ Integrating both sides of (\ref{heatequation-1}) and noting that
$$\left(\frac{1}{|V|}\int_Vu^pd\mu\right)^{1/p}$$
is non-decreasing in $p>1$, we have
\begin{align*}
\frac{d}{dt}\int_Vud\mu&=\int_V\Delta ud\mu+\int_Vu^pd\mu-\int_Vaud\mu\\
&\geq\int_Vu^pd\mu-a_0\int_Vud\mu\\
&\geq|V|\left(\frac{1}{|V|}\int_Vud\mu\right)^p-a_0\int_Vud\mu\\
&=|V|^{1-p}\left(\int_Vud\mu\right)^p-a_0\int_Vud\mu,
\end{align*}
where $a_0=\max_{x\in V}a(x)$.
Denote $z(t)=\int_Vu(\cdot,t)d\mu$, $z_0=z(0)=\int_Vu_0d\mu$ and $a_1=|V|^{1-p}$. The above differential inequality gives
$$z^\prime(t)\geq a_1z^p(t)-a_0z(t),\quad t\in[0,T),$$
which can be solved as follows
\begin{equation}\label{geq}z^{p-1}(t)\geq\frac{1}{(z_0^{1-p}-\frac{a_1}{a_0})e^{(p-1)a_0t}+\frac{a_1}{a_0}}.\end{equation}
If $z_0^{1-p}<a_1/a_0$, or equivalently $z_0>(a_1/a_0)^{1/(p-1)}$, then (\ref{geq}) implies
$$\left(\frac{a_1}{a_0}-z_0^{1-p}\right)e^{(p-1)a_0t}<\frac{a_1}{a_0}.$$
Hence
$$t\leq \frac{1}{(p-1)a_0}\log\frac{1}{1-\frac{a_0}{a_1}z_0^{1-p}},$$
and thus
$$T_{\rm max}(u_0)\leq \frac{1}{(p-1)a_0}\log\frac{1}{1-\frac{a_0}{a_1}z_0^{1-p}},$$
as we desired.

$(ii)$ 
Multiplying both sides of (\ref{heatequation-1}) by $\varphi$ and integrating by parts, we have
\begin{align}\nonumber
\frac{d}{dt}\int_V\varphi ud\mu&=\int_V\varphi\Delta ud\mu-\int_Va\varphi u d\mu+\int_V\varphi u^pd\mu\\
&=\int_V(\Delta\varphi-a\varphi)ud\mu+\int_V\varphi u^pd\mu\nonumber\\
&=-\lambda_a\int_V\varphi ud\mu+\int_V\varphi u^pd\mu.\label{deriv}
\end{align}
Since $\varphi>0$ on $V$, it follows from the H\"older's inequality that
$$
\frac{\int_V\varphi ud\mu}{\int_V\varphi d\mu}\leq \frac{(\int_V\varphi u^pd\mu)^{1/p}(\int_V\varphi d\mu)^{1-1/p}}{\int_V\varphi d\mu}
=\left(\frac{\int_V\varphi u^pd\mu}{\int_V\varphi d\mu}\right)^{1/p}.
$$
Hence
\begin{equation}\label{up}
\int_V\varphi u^pd\mu\geq (\int_V\varphi d\mu)^{1-p}(\int_V\varphi u d\mu)^p.
\end{equation}
Setting $y(t)=\int_V\varphi(x) u(x,t)d\mu$, $y_0=\int_V\varphi u_0d\mu$, and inserting (\ref{up}) into (\ref{deriv}), one derives
$$y^\prime(t)\geq b_1y^p(t)-b_0y(t),$$
where $b_0=\lambda_a$ and $b_1=\left(\int_V\varphi d\mu\right)^{1-p}$.
This leads to an analog of (\ref{geq}),
\begin{equation}\label{geq-1}
y^{p-1}(t)\geq\frac{1}{(y_0^{1-p}-\frac{b_1}{b_0})e^{(p-1)b_0t}+\frac{b_1}{b_0}}.\end{equation}
If $y_0^{1-p}<b_1/b_0$, or equivalently $y_0>(b_1/b_0)^{1/(p-1)}$, then (\ref{geq-1}) implies
$$\left(\frac{b_1}{b_0}-y_0^{1-p}\right)e^{(p-1)b_0t}<\frac{b_1}{b_0}.$$
Hence
$$t\leq \frac{1}{(p-1)b_0}\log\frac{1}{1-\frac{b_0}{b_1}y_0^{1-p}},$$
and thus
$$T_{\rm max}(u_0)\leq \frac{1}{(p-1)b_0}\log\frac{1}{1-\frac{b_0}{b_1}y_0^{1-p}}.$$
This is exactly (\ref{T-2}).
$\hfill\Box$\\


{\emph{Proof of Theorem \ref{f-time1}}.} 
Assume $u:V\times[0,T)\rightarrow\mathbb{R}$ is a solution of the equation (\ref{heatequation-1}) with $u_0\geq 0$. Then, along this flow,
an integration by parts gives
\begin{align*}
\frac{d}{dt}J(u(\cdot,t))&=\int_V(\nabla u\nabla u_t+auu_t)d\mu-\int_Vu^pu_td\mu\\
&=\int_V(-\Delta u+au-u^p)u_td\mu\\
&=-\int_Vu_t^2d\mu.
\end{align*}
This implies the flow (\ref{heatequation-1}) is a negative gradient flow of the energy functional $J$.
In particular,
\begin{equation}\label{decreasing}J(u(\cdot,t))\,\,\, {\rm is\,\, decreasing\,\, in}\,\,\,t\in[0,T).\end{equation}

Multiplying both sides of (\ref{heatequation-1}) by $u$ and applying (\ref{decreasing}) and the H\"older's inequality, we arrive at
\begin{align*}
\frac{d}{dt}\int_Vu^2d\mu&=-2\int_V(|\nabla u|^2+au^2)d\mu+2\int_Vu^{p+1}d\mu\\
&=-4J(u(\cdot,t))+\frac{2p-2}{p+1}\int_Vu^{p+1}d\mu\\
&\geq-4J(u_0)+\frac{2p-2}{p+1}|V|^{\frac{1-p}{2}}\left(\int_Vu^2d\mu\right)^{\frac{p+1}{2}}.
\end{align*}
Set $w(t)=\int_Vu^2(\cdot,t)d\mu$ for $t\in[0,T)$, $w_0=\int_Vu_0^2d\mu$, $d_0=4J(u_0)$,
$d_1=\frac{2p-2}{p+1}|V|^{\frac{1-p}{2}}$ and $\alpha=\frac{p+1}{2}$.
Then the above differential inequality reads as
\begin{equation}\label{w}w^\prime(t)\geq -d_0+d_1w^{\alpha}(t).\end{equation}
If $d_0<0$, then $w_0>0$ and $w^\prime(t)\geq d_1w^{\alpha}(t)$. It follows that $(w^{1-\alpha})^\prime\leq (1-\alpha)d_1$, and that
$$w(t)\geq \left(\frac{1}{w_0^{1-\alpha}+(1-\alpha)d_1t}\right)^{1/(\alpha-1)}.$$
Hence
\begin{equation}\label{T}T_{\rm max}(u_0)\leq\frac{w_0^{1-\alpha}}{(\alpha-1)d_1}.
\end{equation}
If $d_0\geq 0$ and $d_1w_0^{\alpha}>d_0$, then we conclude that $w(t)$ is increasing in $t$ and
$$w^\prime(t)\geq \eta_0=d_1w_0^{\alpha}-d_0>0.$$
The Newton-Lebnitz formula leads to $w(t)\geq w_0+\epsilon_0t$. As long as
$$t>\frac{(\frac{2d_0}{d_1})^{1/\alpha}-w_0}{\eta_0},$$
there holds $\frac{1}{2}d_1w^\alpha(t)>d_0$. Inserting this into (\ref{w}), we have the inequality $w^\prime(t)\geq
\frac{d_1}{2}w^{\alpha}(t)$
and an analog of (\ref{T})
$$T_{\rm max}(u_0)\leq\frac{2w_0^{1-\alpha}}{(\alpha-1)d_1}.$$
Therefore
$$T_{\rm max}(u_0)\leq \max\left\{\frac{(2d_0/d_1)^{1/\alpha}-w_0}{\eta_0},\frac{2w_0^{1-\alpha}}{(\alpha-1)d_1}\right\}.$$
This ends the proof of $(iii)$. 
$\hfill\Box$

Next, we establish another criterion that guarantees blow-up when the initial condition exceeds a positive equilibrium.

\begin{proposition}\label{equi}
Suppose that the problem (\ref{heatequation-1}) has an equilibrium $v$ satisfying $v>0$ on $V$. If $u_0\geq v$ and
$u_0\not\equiv v$, then $T_{\rm max}(u_0)<+\infty$.
\end{proposition}

\begin{proof}
Note that
\begin{equation}\label{bril}\Delta v-av+v^p=0\end{equation}
and $0<v\leq u_0$ in $V$.
By the comparison principle (Lemma \ref{comparison-1}), we have
$$u(x,t)\geq v(x),\quad\forall (x,t)\in V\times[0,T_{\rm max}(u_0)).$$
Since $u_0\not\equiv v$, by the continuity, there exists a sufficiently small $\tau_0>0$ such that for all $t\in[0,\tau_0]$,
there holds $u(\cdot,t)\not\equiv v(\cdot)$.
Let $\tau\in(0,\tau_0]$ be fixed. We {\it claim} that
\begin{equation}\label{tau}
u(x,\tau)>v(x),\quad\forall x\in V.
\end{equation}
For otherwise, there exists some $x_0\in V$ such that
$$u(x_{0},\tau)-v(x_0):=(u-v)(x_0,\tau)=\min_{x\in V}(u-v)(x,\tau)=0.$$
 Hence
\begin{align*}
0\geq \partial_t(u-v)(x_0,\tau)&=\Delta (u-v)(x_0,\tau)-a(x_0)(u-v)(x_0,\tau)+u^p(x_0,\tau)-v^p(x_0)\\
&=\Delta (u-v)(x_0,\tau)\geq 0.
\end{align*}
This leads to $u(x,\tau)=v(x)$ for all $x\sim x_0$. Repeating this process, we have $u(x,\tau)=v(x)$ for all $x\in V$,
contradicting the fact $u(\cdot,\tau)\not\equiv v$ on $V$. Hence our claim (\ref{tau}) follows. It then follows that
$$u(x,\tau)\geq \beta v(x),\quad\forall x\in V$$
for some $\beta>1$. In view of (\ref{bril}),
$$\partial_t(\beta v)-\Delta (\beta v)+a(\beta v)-(\beta v)^p<-\beta\Delta v+a\beta v-\beta v^p=
\partial_t u-\Delta u+au-u^p$$
 Again, by the comparison principle, we obtain
 \begin{equation}\label{strict}u(x,t)\geq \beta v(x),\quad\forall (x,t)\in V\times [\tau,T_{\rm max}).\end{equation}
 Set $w=w(t)=\int_Vu(\cdot,t)vd\mu$. Multiplying the equation (\ref{heatequation-1}) by $v$, integrating by parts,
 and using (\ref{strict}) and the H\"older's inequality, we have for $t\in[\tau,T_{\rm max}(u_0))$,
 \begin{align*}
 w^\prime&=\int_Vu_tvd\mu=\int_Vu\Delta vd\mu+\int_V(u^p-au)vd\mu\\
 &=\int_Vu(av-v^p)d\mu+\int_V(u^p-au)vd\mu\\
 &=\int_V\left(1-\left(\frac{v}{u}\right)^{p-1}\right) u^pvd\mu\\
 &\geq(1-\beta^{1-p})\left(\int_Vvd\mu\right)^{1-p}w^p.
 \end{align*}
 Hence
 $$\frac{1}{1-p}w^{1-p}(t)-\frac{1}{1-p}w^{1-p}(\tau)\geq (1-\beta^{1-p})\left(\int_Vvd\mu\right)^{1-p}(t-\tau),$$
 which together with $\lim_{t\rightarrow T_{\rm max}(u_0)}w(t)=+\infty$ gives
 $$T_{\rm max}(u_0)\leq \tau+\frac{w^{1-p}(\tau)}{(p-1)(1-\beta^{1-p})}\left(\int_Vvd\mu\right)^{p-1}.$$
 This ends the proof of the theorem. 
 \end{proof}

 For the blow-up rate, we have the following result.

 \begin{proposition}\label{speed}
Let $u:[0,T)\rightarrow \mathbb{R}$ be a solution of (\ref{heatequation-1}) with $u_0\geq 0$ and $\lim_{t\rightarrow T-0}\|u\|_{L^\infty(V)}=+\infty$.
Then there holds for all $t\in[0,T)$,
$$\lim_{t\rightarrow T-0}(T-t)\left(\max_{x\in V}u(x,t)\right)^{p-1}=\frac{1}{p-1}.$$
\end{proposition}

\proof Let $u(x,t)$ be a solution of (\ref{heatequation-1}) existing in the interval $[0,T)$. Let
\begin{equation}\label{eq-6}\phi(t)=\max_{x\in V}u(x,t)=u(x_t,t),\end{equation}
where $x_t\in V$ is one of the maximum points of $\phi$ for each $t\in[0,T)$. We {\it claim} that $\phi$
is a locally Lipschitz function, in particular, for almost every $t\in[0,T)$, $\phi$ is differentiable
at $t$ and
\begin{equation}\label{der-1}\phi^\prime(t)=(\partial_tu)(x_t,t).\end{equation}
Indeed, for any $t_1,t_2\in[0,T)$, there exist two points $x_{t_1}$ and $x_{t_2}\in V$ such that $\phi(t_1)=u(x_{t_1},t_1)$ and
$\phi(t_2)=u(x_{t_2},t_2)$. On one hand,
\begin{align}\nonumber
\phi(t_1)-\phi(t_2)&=u(x_{t_1},t_1)-u(x_{t_2},t_2)\\\nonumber
&=u(x_{t_1},t_1)-u(x_{t_1},t_2)+u(x_{t_1},t_2)-u(x_{t_2},t_2)\\\nonumber
&\leq(\partial_tu)(x_{t_1},\xi)(t_1-t_2)\\
&\leq|(\partial_tu)(x_{t_1},\xi)||t_1-t_2|,\label{eq-7}
\end{align}
since $u(x_{t_1},t_2)\leq u(x_{t_2},t_2)$, where $\xi$ lies between ${t_1}$ and ${t_2}$. On the other hand,
\begin{align}\nonumber
\phi(t_2)-\phi(t_1)&=u(x_{t_2},t_2)-u(x_{t_1},t_1)\\\nonumber
&=u(x_{t_2},t_2)-u(x_{t_2},t_1)+u(x_{t_2},t_1)-u(x_{t_1},t_1)\\\nonumber
&\leq(\partial_tu)(x_{t_2},\eta)(t_2-t_1)\\
&\leq|(\partial_tu)(x_{t_2},\eta)||t_1-t_2|,\label{eq-8}
\end{align}
since $u(x_{t_2},t_1)\leq u(x_{t_1},t_1)$, where $\eta$ lies between ${t_1}$ and ${t_2}$. Hence we conclude that
$\phi(t)$ is locally Lipschitz in $t\in[0,T)$. In particular, $\phi^\prime(t)$ is differentiable almost everywhere in
$[0,T)$. To prove (\ref{der-1}), we fix any $t\in[0,T)$ such that $\phi^\prime(t)$ exists. For any $h>0$ with $t+h<T$, we have by an analog of (\ref{eq-7}),
$$\frac{\phi(t+h)-\phi(t)}{h}\geq(\partial_tu)(x_t,t+\theta h),\quad \theta\in(0,1).$$
Letting $h\rightarrow 0+$, we get $\phi^\prime(t)\geq (\partial_tu)(x_t,t)$. Similarly, if $0<h<t$, then an analog of (\ref{eq-8}) leads to
$$\frac{\phi(t-h)-\phi(t)}{-h}\leq (\partial_tu)(x_{t},t-\tau h),\quad \tau\in(0,1).$$
It follows that $\phi^\prime(t)\leq (\partial_tu)(x_t,t)$. Thus $\phi^\prime(t)=(\partial_tu)(x_t,t)$ and our claim follows.

We now proceed to prove the theorem. On one hand, if $\phi^\prime(t)$ exists, then we have
\begin{align*}
\phi^\prime(t)&=(\partial_tu)(x_t,t)\\
&=\Delta u(x_t,t)-a(x_t)u(x_t,t)+u^p(x_t,t)\\
&\leq\phi^p(t).
\end{align*}
Integrating both sides of the above differential inequality on the time interval $[t,T)$, we obtain
\begin{equation}\label{leq}
(T-t)\phi^{p-1}(t)\geq\frac{1}{p-1}.
\end{equation}
On the other hand, $\forall (x,t)\in V\times[0,T)$, one calculates
\begin{align*}
\partial_t u(x,t)&=\Delta u(x,t)-a(x)u(x,t)+u^p(x,t)\\
&\geq-\frac{\sum_{y\sim x}\omega_{xy}}{\mu(x)}u(x,t)-a(x)u(x,t)+u^p(x,t)\\
&\geq-Au(x,t)+u^p(x,t),
\end{align*}
where
$$A=\max_{z\in V}\left\{\frac{\sum_{y\sim z}w_{zy}}{\mu(z)}+a(z)\right\}.$$
Hence for almost every $t\in[0,T)$, there holds
$$\phi^\prime(t)\geq -A\phi(t)+\phi^p(t).$$
As a consequence, one has
\begin{equation*}
(T-t)\phi^{p-1}(t)\leq \frac{A(T-t)}{1-e^{-(p-1)A(T-t)}},
\end{equation*}
and thus
$$\limsup_{t\rightarrow T-0}\,(T-t)\phi^{p-1}(t)\leq \frac{1}{p-1}.$$
This together with (\ref{leq}) gives the desired result. $\hfill\Box$

\section{Applications in complex dynamical networks}

For simplicity, we consider a weighted and connected network $G_{25}$ consisting of $25$ nodes (Figure \ref{fig1}), whose states are described by $1$-dimensional scalar functions $u_i(t)$, i.e., we set $n=1$. We also take $\alpha=1$, $H=1$ and set $f(u_i)=|u_i|^{p-1}u_i$ in \eqref{system1}. Obviously, $f(\cdot)$ is not a Lipschitz function. The controller $a_i$ is chosen as $a(u_i-\bar{u})$, where $a$ is a positive constant. At this stage, system \eqref{system1} reduces to
\begin{equation}\label{equation}
	\dot{u_i} =|u_i|^{p-1}u_i+\sum_{j=1}^{N}l_{ij}u_j+a_i(u_i-\bar{u}),
\end{equation}
where $u_i=u(x_i,t)$ and $x_i\in V$.
When considering the evolution of such a nonlinear system \eqref{system1}, we did not use the conventional methods typically applied in the study of ordinary differential systems. Instead, we treat it as a partial differential equation on a graph, which is equivalent to studying the following nonlinear heat equation on the graph
\begin{equation}\label{heat}
	\partial_t u =\Delta u-a(u-\overline{u})+|u|^{p-1}u,
\end{equation}
where $u: V\times [0,+\infty)\rightarrow \mathbb{R}$ and $u(x_i,t)=u_i(t)$ is a scalar function representing the state of node $x_i$. If we take the equilibrium state as $\overline{u}=0$, one immediately sees that equation \eqref{heat} coincides with equation \eqref{heatequation-1} in our consideration. 

We set the measure at each node $x_i\in V$ to $\mu(x_i)=1$ and the weight on each edge $e_{ij}\in E$ is also set to $\omega_{ij}=1$. In the network $G_{25}$, to highlight the central position of node $x_1$, we set the initial values $u_i(0)=0.001$ and $a_i=2$ for $i\neq 1$, while assigning $x_1$ an initial value of $0.03$ with $a_1=0$. First, we simulate (ii) of Theorem \ref{global} to demonstrate how the initial values $u_0$ and the first eigenvalue $\lambda_a$ govern the network's convergence to the zero equilibrium when long-time solutions exist. Here we set $p=2$. For the network in Figure \ref{fig1}, $\lambda_a=1.9116$ and $\|u_0\|_{L^2(V)}=0.0304<\epsilon_0=0.0365$. We solve the system using MATLAB's ode45 solver and the results are presented in Figure \ref{fig2}, which includes: (i) dynamic curves of four representative nodes $x_1$, $x_2$, $x_{20}$ and $x_{21}$, (ii) the decay-rate control curve determined by initial values and the first eigenvalue, i.e., $\frac{1}{\sqrt{\mu_{\rm min}}}\|u_0\|_{L^2(V)}\,e^{-\frac{\lambda_a}{2} t}$, and (iii) magnified views of the aforementioned curves.
 
  \begin{figure}[H]
	\centering
	\includegraphics[width=12cm]{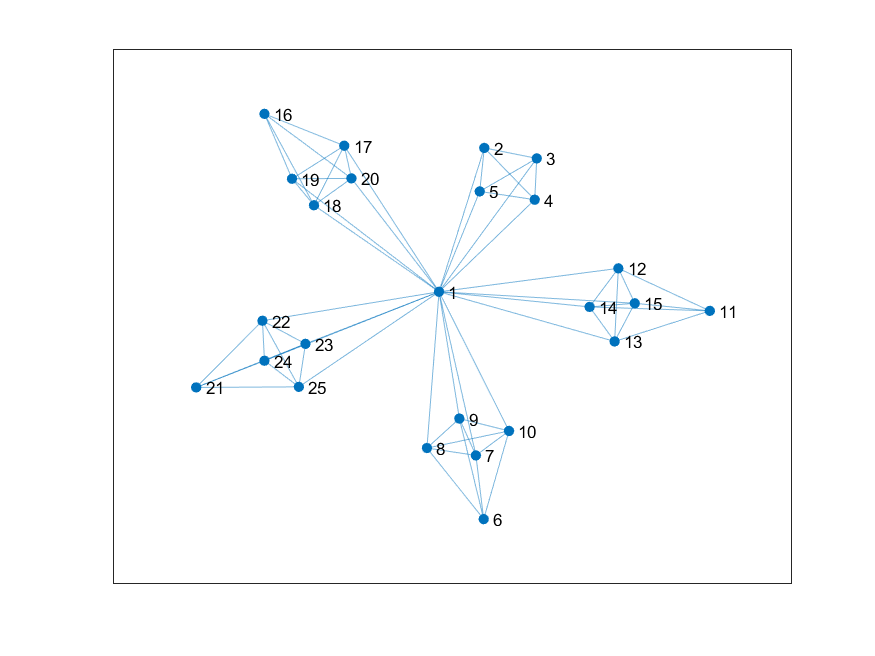}
	\caption{The network $G_{25}$ with $25$ nodes}\label{fig1}
\end{figure}

\begin{figure}[H]
	\centering
	\includegraphics[width=12cm]{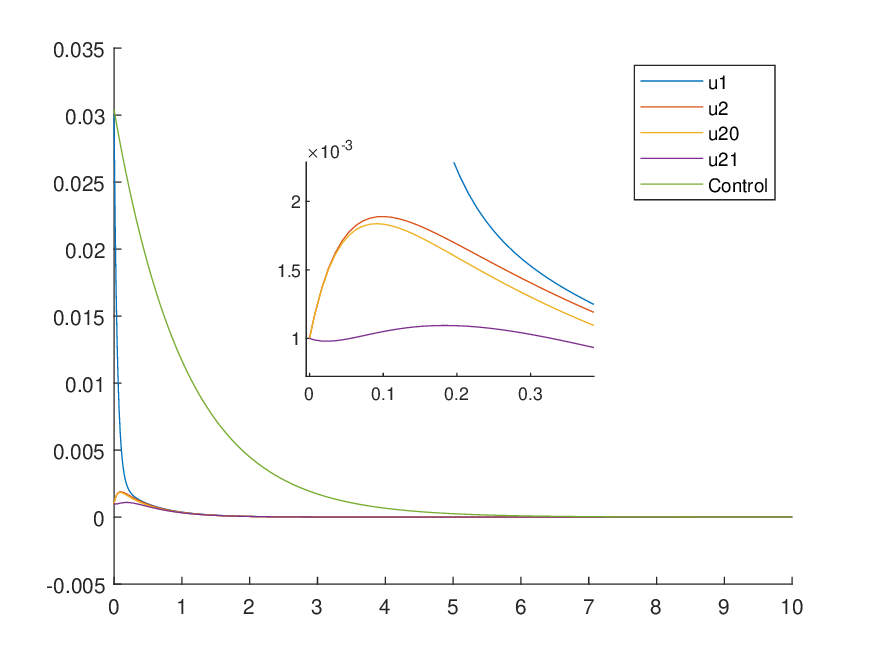}
	\caption{Dynamic curves of $G_{25}$}\label{fig2}
\end{figure}
  
Next, we proceed to verify part (i) of Theorem \ref{finitetime}, demonstrating how the relationship between initial values and the controller $\mathbf{a}_i$ enables finite-time blow-up of the dynamical network. For the network $G_{25}$ in Figure \ref{fig1}, we still set the measure at each node $x_i\in V$ to $\mu(x_i)=1$ and the weight on each edge $e_{ij}\in E$ is set to $\omega_{ij}=1$. We take $p=3$, $a_i=2$ for $i\neq 1$, and $a_1=0$. The initial value $u_1(0)=0.5$, while initial values of other nodes is set to $1.5$. Direct calculation yields
$$\int_Vu_0d\mu=36.5000>\left(\max_{x\in V}a(x)\right)^{{1}/{(p-1)}}|V|=35.3553.$$
According to Theorem \ref{finitetime}, the system will experience blow-up at time $T_{\rm max}(u_0)$, which occurs earlier than
\begin{equation*}
	\frac{1}{(p-1)\max_{x\in V}a(x)}\log\frac{1}{1-|V|^{p-1}\max_{x\in V}a(x)
		\left(\int_Vu_0d\mu\right)^{1-p}}=0.6962
\end{equation*}
We still solve the system using MATLAB's ode45 solver and the results of dynamic curves of node $x_1$ and $x_9$ (curves of other nodes are similar to that of node $x_9$) are presented in Figure \ref{fig3}. As visible from the figure, the system experiences blow-up before $t=0.0045$.

\begin{figure}[H]
	\centering
	\includegraphics[width=12cm]{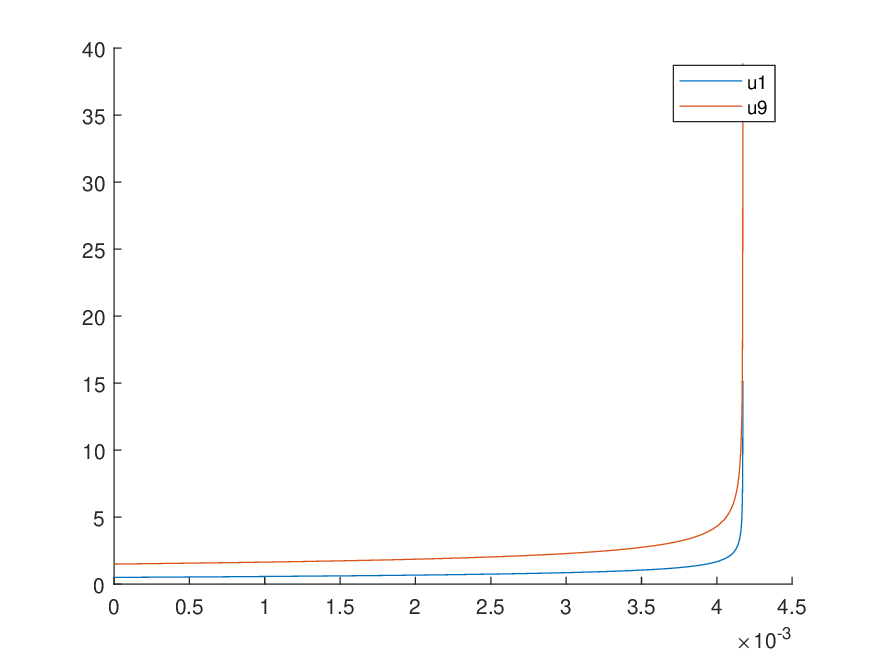}
	\caption{Blow-up of $G_{25}$}\label{fig3}
\end{figure}

\begin{remark}
	Although the estimates in Theorems \ref{global} and \ref{finitetime} regarding long-time existence or blow-up of solutions are not sharp, they nevertheless provide a feasible approach to steer the complex dynamical networks toward equilibrium or blow-up. Specifically, Theorem \ref{global} guarantees that if the initial value is smaller than a constant $\epsilon_0$ (determined by structure parameters of the network), the solution exists for all time and converges exponentially to the $\mathbf{u}=\mathbf{0}$. Theorem \ref{finitetime} asserts that blow-up inevitably occurs when the control parameter $\mathbf{a}_i$ is sufficiently small (or the initial value sufficiently large) to satisfy the blow-up condition stated in the theorem.
\end{remark}

\section*{Acknowledgements}

This research is supported by National Natural Science Foundation of China (No. 12271039 and No. 12101355), the National Key R and D Program of
China (No. 2020YFA0713100), the Open Project Program (No. K202303) of Key Laboratory of
Mathematics and Complex Systems, Beijing Normal University and the Youth Innovation Science and Technology Support Program for Universities in Shandong Province (No. 2024KJG008).

\bibliographystyle{elsarticle-num-names-alpha}

\bibliography{mybib-graph}

\end{document}